\documentclass[12pt,reqno]{amsart}
\usepackage{amsmath, amsthm, amsopn, amssymb, microtype, bm}
\usepackage{mathrsfs} 
\usepackage{mathscinet}
\usepackage{bbm}
\usepackage{enumerate,tikz, etoolbox, intcalc, geometry, caption, subcaption, afterpage,quiver}
\usepackage[english]{babel}
\usepackage{enumitem}

\usepackage[colorlinks=true,urlcolor=blue,pdfborder={0 0 0}]{hyperref}
\hypersetup{linkcolor=[rgb]{0,0,0.6}}
\hypersetup{citecolor=[rgb]{0,0.6,0}}
\usepackage[nameinlink]{cleveref}

\theoremstyle{definition}
\newtheorem{definition}{Definition}[section]
\theoremstyle{remark}
\newtheorem{example}[definition]{Example}

\newtheorem{remark}[definition]{Remark}

\theoremstyle{plain}

\newtheorem{theorem}[definition]{Theorem}
\newtheorem*{theorem*}{Theorem}

\newtheorem{lemma}[definition]{Lemma}

\def\ol{\overline}
\def\ov{\overline}

\newcommand{\N}{\mathbb N}

\newcommand{\be}{\begin{equation}}
\newcommand{\ee}{\end{equation}}
\newcommand{\ba}{\begin{aligned}}
\newcommand{\ea}{\end{aligned}}
\newcommand{\mc}{\mathcal}
\newcommand{\ignore}[1]{}

\numberwithin{equation}{section}

\title{Generalized Bratteli diagrams without probability tail invariant measures}

\author[Barsky]{Andrew Barsky}
\address{Department of Mathematics, University of Iowa, Iowa City, IA 52242-1419, USA}
\email{andrew.barsky@doctoral.uj.edu.pl}

\author[Bezuglyi]{Sergey Bezuglyi}
\address{Department of Mathematics, University of Iowa, Iowa City, IA 52242-1419, USA}
\email{sergii-bezuglyi@uiowa.edu}



\date{}

\begin{document}

\begin{abstract}
The paper focuses on the following problem: given a generalized Bratteli diagram $B$, determine conditions under which the tail equivalence relation $\mathcal R$ does not admit a tail invariant probability measure on the path space $X_B$. We give several sufficient conditions, formulated in terms of the incidence matrices, that guarantee the nonexistence of such a measure.
 We also discuss a new spectral method for finding tail invariant measures for a class of stationary generalized Bratteli diagrams. This method allowed us
to answer affirmatively an open problem formulated in \cite{BezuglyiJorgensenKarpelSanadhya2026}. 
\end{abstract}

\maketitle
\tableofcontents

\section{Introduction}\label{sect Intro} 

The paper is devoted to the problem of the existence of invariant probability measures for Borel automorphisms of a standard Borel space.
 In \cite{BezuglyiDooleyKwiatkowski2006}, it was proved that every aperiodic Borel automorphism is isomorphic to a Bratteli-Vershik transformation defined on the path space of a generalized Bratteli diagram (see Section
\ref{sect BD} for definitions). 
In other words, we can use generalized Bratteli diagrams to construct isomorphic models of Borel automorphisms. For standard Bratteli diagrams (with a finite number of vertices at each level), invariant probability measures always exist, and their description is closely tied to finite-dimensional nonnegative matrix theory.
For generalized Bratteli diagrams (with countable sets of vertices at each level), this picture changes substantially: the incidence matrices are infinite, the path space need not be locally compact, and invariant measures, when they exist, need not be finite. This raises a basic question: how can one detect directly from the incidence matrices that no invariant probability measure exists?

The existence of invariant probability measures for Borel automorphisms and countable Borel equivalence relations is a classical problem in Borel dynamics; see, for example, \cite{Nadkarni1991, Dougherty_Jackson_Kechris1994, KechrisMiller2004, Kechris2024}. In particular, Nadkarni's theorem characterizes the nonexistence of an invariant probability measure for an aperiodic Borel automorphism in terms of compressibility. The purpose of this paper is to study this problem for Borel dynamical systems represented by generalized Bratteli diagrams and to obtain explicit criteria for nonexistence in terms of their incidence matrices.

\textit{Bratteli diagrams} provide a versatile framework for representing dynamical systems through graded graphs and their incidence matrices. In the classical setting, ordered Bratteli diagrams model homeomorphisms of Cantor spaces via Vershik maps, connecting topological dynamics, orbit equivalence, and dimension groups \cite{HermanPutnamSkau1992, Durand2010, BezuglyiKarpel2016}. Generalized Bratteli diagrams, whose vertex sets are countably infinite at every level, extend this framework to non-compact zero-dimensional Polish spaces and to the study of \textit{Borel automorphisms}.
Their path spaces carry a natural tail equivalence relation, which is a hyperfinite countable Borel equivalence relation and, under suitable order assumptions, may be generated by a Borel Vershik map; see   \cite{HermanPutnamSkau1992, GiordanoPutnamSkau1995, Medynets_2006, BezuglyiKwiatkowskiYassawi2014, BezuglyiJorgensenKarpelSanadhya2026} for further discussion of Vershik maps.

Invariant measures are central to the ergodic theory of Bratteli-Vershik systems and countable Borel equivalence relations. For a Bratteli diagram, invariance under the tail equivalence relation means that cylinder sets determined by finite paths with the same terminal vertex have equal measure. Such measures are encoded by sequences of nonnegative vectors 
$p^{(n)}$ (infinite-dimensional for generalized Bratteli diagrams) 
satisfying compatibility relations $A_n p^{(n+1)} = p^{(n)}$, determined by the transposes $A_n$ of the incidence matrices $F_n$. In the stationary finite-rank setting, this correspondence leads to a well-developed theory based on nonnegative matrices and their eigenvectors \cite{BezuglyiKwiatkowskiMedynetsSolomyak2010, BezuglyiKarpel2016, BezuglyiJorgensen2022}. For generalized Bratteli diagrams, however, the incidence matrices are infinite, the path space is generally non-locally compact, and the existence of a probability tail invariant measure is no longer automatic. Even when nonzero invariant measures exist, they may be infinite rather than finite.

This paper addresses the following problem: given a generalized Bratteli diagram $B$, determine verifiable conditions under which its path space $X_B$ admits no invariant probability measure under the tail equivalence relation. Equivalently, when the relation is generated by a Borel Vershik map, the problem asks when the associated Borel automorphism has no invariant probability measure of this type. Thus, any nonzero tail invariant measure, if it exists, must be infinite. The aim is to formulate nonexistence criteria directly in terms of the incidence matrices of the diagram and the resulting numbers of finite paths. 

A related notion in Borel dynamics is that of \textit{compressibility}. A countable Borel equivalence relation is called compressible when, in a Borel sense, its space can be embedded into a proper part of itself along equivalence classes. Compressibility is therefore an appropriate dynamical expression of the escape-of-mass phenomenon: the system can be shifted into a smaller Borel subset without leaving the equivalence classes. A fundamental principle in the theory of countable Borel equivalence relations asserts that compressibility is incompatible with the existence of an invariant probability measure \cite{Nadkarni1991, Nadkarni1995, Kechris2024}. In the setting of generalized Bratteli diagrams, the tail relation is a hyperfinite aperiodic countable Borel equivalence relation, and when an appropriate Vershik map exists, it can be viewed as the orbit relation of a Borel automorphism. Thus, the nonexistence of a probability tail invariant measure may also be interpreted as evidence of a compressible Borel dynamical structure. The incidence-matrix criteria developed in the paper provide concrete combinatorial conditions for detecting this phenomenon.

Let $h_{v,w}^{(n)}$ denote the number of finite paths from a vertex $w$ at level zero to a vertex $v$ at level $n$, and let $H_v^{(n)}$ denote the total number of finite paths terminating at $v$. The ratio
$ \dfrac{h_{v,w}^{(n)}}{H_v^{(n)}} $
is the proportion of finite paths that end in $v$ and originate in $w$. More generally, sums of these ratios describe the proportion of paths originating in a prescribed finite collection of vertices. These quantities provide a natural way to detect the escape of mass across the countably infinite levels of a generalized Bratteli diagram. 
The decay of the quantities $\dfrac{h_{v,w}^{(n)}}{H_v^{(n)}}$, uniformly in suitable vertices $v,w$, can therefore be viewed as an indication of the escape of mass.

Our \textit{main results} are divided into two groups.
The first group of results establishes sufficient conditions for the nonexistence of invariant measures in the probability tail based on the asymptotic behavior of the path ratios $\frac{h_{v,w}^{(n)}}{H_v^{(n)}}$. In particular, if the contribution of every fixed initial vertex becomes uniformly negligible among paths terminating at level $n$, then no probability tail invariant measure can exist. 

\begin{theorem} Let $B = B(F_n)$ be a generalized Bratteli diagram defined by a sequence of incidence matrices $(F_n)$.
Suppose that for all $w\in V_{0}$,
\begin{equation*}    
\liminf_{n\to\infty}\left(\sup_{v\in V_{n}}\frac{h_{v,w}^{(n)}}{H_{v}^{(n)}}\right)=0
\end{equation*}
Then $B$ does not admit a probability tail invariant measure. Moreover, 
this condition is sufficient but not necessary. 
\end{theorem}

This principle is also developed for finite subsets of vertices at an arbitrary fixed level. The proofs combine the normalization identity for a probability tail invariant measure with estimates derived from finite-path counts. The resulting criteria apply to both stationary and nonstationary diagrams and are expressed entirely through products of incidence matrices. We also consider balanced path growth and a bounded-distortion condition that allows asymptotic information along selected vertices to be promoted to estimates uniform across an entire level.
 These criteria reflect a common mechanism: the failure of probability invariance can be viewed as an \textit{escape of mass} encoded combinatorial by the incidence matrices.
\\

The second part of the paper considers a class of stationary generalized Bratteli diagrams whose incidence matrix $A$ satisfies the following conditions:
\begin{itemize}
    \item $A$ is a nonnegative symmetric incidence operator;
    \item $A$ has a complete orthonormal family of eigenvectors in the space $\ell^2(V_0)$;
    \item  the tower heights have a uniform lower bound that dominates every fixed exponential rate.
\end{itemize}

An example of such a matrix is 
$$
F =
\begin{pmatrix}
2&1&0&0&\cdots\\
1&3&1&0&\cdots\\
0&1&4&1&\cdots\\
0&0&1&5&\ddots\\
\vdots&\vdots&\vdots&\ddots&\ddots 
\end{pmatrix}.
$$
This diagram consists of a countable family of odometer subdiagrams with increasing multiplicities, coupled by single edges between neighboring indices. The question of the existence of a probability tail invariant measure for this diagram was posed in \cite{BezuglyiJorgensenKarpelSanadhya2026}. We prove here that no such measure exists. 

Analyzing the properties of the matrix $F$, we treat it as an unbounded Jacobi-type operator on $\ell^2(\N)$. We establish
that this operator is self-adjoint and has compact resolvent; consequently, it has purely
discrete spectrum and an orthonormal basis of eigenvectors. At the same time, a direct
combinatorial argument gives a uniform superexponential lower bound for the heights of the
Kakutani-Rokhlin towers. If a probability tail invariant measure existed, its defining vectors
would satisfy the incidence-matrix recursion. Pairing this recursion with the eigenvectors of
the operator shows that every Fourier coefficient of the initial measure vector must vanish:
exponential growth along each fixed eigenvector is dominated by the uniform growth of the
tower heights. Completeness of the eigenvectors then forces the initial vector to be zero,
contradicting probability normalization.

The proof of this result introduces a new method different from all the methods known before. We call it the \textit{spectral method}. 
Let us recall what main methods have been used to study tail invariant measures. If $B=B(F)$ is a stationary Bratteli diagram, then Perron-Frobenius theory (for finite and infinite nonnegative matrices $F$) can be used to find explicit formulas for the values of tail invariant measures on cylinder sets in terms of the Perron eigenvalue and the corresponding nonnegative eigenvector. 
This method has been used in 
\cite{BezuglyiKwiatkowskiMedynetsSolomyak2010, BezuglyiKarpelKwiatkowski2019, BezuglyiKarpel_2020, BezuglyiJorgensenKarpelSanadhya2026} and other papers. The second method consists of finding a subdiagram of $B$ that supports a probability tail invariant measure and extending this measure by tail equivalence to the path space of $B$. The main difficulty in the realization of this construction is to determine when the extended measure is finite; see also the papers 
\cite{BezuglyiKwiatkowskiMedynetsSolomyak2013,  AdamskaBezuglyiKarpelKwiatkowski2017, BezuglyiKarpel2016, BezuglyiJorgensenKarpelSanadhya2026}. More recently, a method based on the computation of the inverse limit of closed convex subsets in an infinite-dimensional space was developed in \cite{BezuglyiKarpelKwiatkowskiWata2024}.

 The \textit{spectral method}
is the principal new idea used in Section \ref{sect 4}. It illustrates how operator-theoretic properties of an infinite incidence matrix can be combined with combinatorial growth estimates to establish nonexistence of invariant probability measures. Unlike the ratio criteria, the argument does not require uniform control of the contribution of each initial vertex. It therefore offers a complementary approach for stationary generalized Bratteli diagrams whose unbounded incidence matrices admit a useful self-adjoint realization.

\begin{theorem}
Let $B = B(F)$ be a stationary generalized Bratteli diagram, with $A=F^T$ realized as a self-adjoint operator acting on $\ell^2(V_0)$. Assume that $A$ has an orthonormal basis $\{\xi^{(k)}\}$ of eigenvectors with finite eigenvalues $\lambda_k$. Define $M_n := \inf_{v\in V_n} H_v^{(n)}$. If 
for every fixed $k$,
$$
\frac{\lambda_k^n}{M_n} \to 0, \quad n\to \infty, 
$$
then $B$ admits no tail invariant probability measure.
\end{theorem}

The paper is organized as follows. Section \ref{sect BD} recalls generalized Bratteli diagrams, their path spaces, tail equivalence relations, Vershik maps, and the vector description of tail invariant measures. Section \ref{sect Ratio test} develops sufficient nonexistence conditions based on ratios of finite-path counts, finite subsets of vertices, balanced growth, and bounded distortion. Section \ref{sect 4} studies a class of stationary Bratteli diagrams. We focus first on a particular case represented by the matrix $F$, see above. Then we extend this method to more general classes of matrices. 
After establishing the spectral properties of its incidence operator and deriving uniform lower bounds for tower heights, we prove that its path space admits no probability tail invariant measure, thereby resolving the problem raised in \cite{BezuglyiJorgensenKarpelSanadhya2026}. 

\section{Basics on Bratteli diagrams}\label{sect BD}

In this section, we introduce the terminology and basic facts that will be used throughout the paper. We begin with the definition of a Bratteli diagram. 

\begin{definition}\label{def BD}
A \textit{Bratteli diagram} is a countably infinite graph $B=(V, E)$ whose vertex and edge sets are partitioned into disjoint subsets
 $V=\bigsqcup_{n=0}^{\infty}V_{n}$, 
$E=\bigsqcup_{n=0}^{\infty}E_{n}$ satisfying the following properties:
\begin{enumerate}[label=(\roman*)]
  \item The edges in $E_n$ connect vertices in $V_n$ to vertices in $V_{n+1}$;

  \item For every edge $e\in E$, let $s(e)$ and $r(e)$ denote its \textit{source} and \textit{range}, respectively. The source and range maps satisfy $s(E_n)=V_n$, $r(E_n)=V_{n+1}$,  $n\in\N_0$.
In particular, $s^{-1}(v)\neq\emptyset$ for every $v\in V$, and $r^{-1}(v)\neq\emptyset$ for every $v\in V\setminus V_0$.

  \item Let  $\ol x = (e_i: e_i\in E_i)$ be a finite or infinite sequence of edges (it is called a \textit{path} in the diagram $B$) such 
that $s(e_i)= r(e_{i-1})$. We denote the set of all infinite
paths $\ol x$ starting at some vertex in $V_0$ by $X_B$ and call 
it the \textit{path space} of the diagram $B$. 
  \end{enumerate} 

If $|V_n| = \aleph_0$ for all $n$ and $|r^{-1}(v)| <\infty$ for every vertex $v \in V \setminus V_0$, then $B$ is called a \textit{generalized Bratteli diagram}. If $|V_n| < \infty$, then we call $B$ a \textit{standard Bratteli diagram}. Here and throughout the paper, we use $|\cdot |$ to denote the cardinality of a set.
\end{definition} 

The structure of every Bratteli diagram $B$ is completely determined by a sequence of nonnegative integer-valued \textit{incidence matrices} $F_n$. For $m > n$, denote by $E(v,w)$ the set of all 
 finite paths between the vertices $w \in V_n$ and
$v \in V_m$ (for some $w, v$ this set may be empty). Clearly $E(v,w) = E(w, v)$, so we will use both notations. 
 For $n \in \N_0$, let $f^{(n)}_{v,w} = |E(v,w)|$ for all $w \in V_n$ and
$v \in V_{n+1}$. We define the incidence matrix $F_n$, $n \in \N_0$, by setting 
\begin{equation}\label{Notation:f^i}
    F_n = (f^{(n)}_{v,w} : v \in V_{n+1}, w\in V_n),\ \   
    f^{(n)}_{v,w}  \in \N_0.
\end{equation}  
We write $B = B(F_n)$ to emphasize that a Bratteli diagram $B$ is defined by the sequence $(F_n)$ of incidence matrices.

For a finite path $\ol e = (e_0, ... , e_n)$, we write
$s(\ol e) = s(e_0)$ and $r(\ol e) = r(e_n)$. The set
$$
    [\ol e] := \{x = (x_i) \in X_B : x_0 = e_0, ..., x_n = e_n\}
$$ 
is called the \textit{cylinder set} associated with $\ol e$. The collection of cylinder sets $[\ol e]$ forms the clopen basis of the topology on the path space $X_B$. For a standard Bratteli diagram, $X_B$ is compact, and for a generalized Bratteli diagram, $X_B$ is a zero-dimensional Polish space. In general, $X_B$ is not locally compact. As usual, we will consider the Bratteli diagrams whose path spaces have no isolated points. 

In what follows, we will work with generalized Bratteli diagrams in general.  

\begin{definition}\label{Def:Tail_equiv_relation}
Two paths $x= (x_i)$ and $y=(y_i)$ in $X_B$ are called 
\textit{tail equivalent} if there exists an $n \in \mathbb{N}_0$ 
such that $x_i = y_i$ for all $i \geq n$. This notion defines a \textit{countable Borel equivalence relation (CBER)} $\mathcal R$ in the path space $X_B$, which is called the \textit{tail equivalence relation}.
We consider the generalized Bratteli diagrams for which $\mc R$ is aperiodic.
\end{definition}

The pair $(X_B, \mc R)$ is a hyperfinite aperiodic Borel dynamical system. This means that this equivalence relation is generated by an aperiodic Borel automorphism of the path space $X_B$. With some natural assumptions, such an automorphism can be realized as a Vershik map $\varphi_B$. 
To define it acting on $X_B$, we need the notion of  
an \textit{ordered Bratteli diagram} $(B, >) =(V, E, >) $. For this, we take a linear order $>$ on  the set $ r^{-1}(v) = \{ e \in E_n : r(e) = v\}$, where $v\in V_{n+1}$ and $n\in \N_0$. Then $>$ defines a partial order on the set $E$.

Let $\ol e= (e_0,e_1,..., e_i,...)$ be an infinite path and $v_{i+1} = r(e_i)$, $i\geq 0$. We say that $\ol e$ is   
\textit{maximal (respectively, minimal)} if every $e_i$ is the maximal, respectively minimal, element of $r^{-1}(v_{i+1})$. The same definition is used 
for finite maximal/minimal paths. Let $X_{max}$ and $X_{min}$ denote the sets of all maximal and minimal paths in $X_B$. It is easy to see that they are closed in $X_B$. Note that there are unique minimal and maximal 
paths in the set $E(V_0, v) = \bigcup_{w\in V_0} E(w,v)$ of all finite paths whose range is 
$v$ for each $v\in V_i,\ i > 0$, where finite paths are compared lexicographically from their terminal end. 

We recall the definition of a (Borel) Vershik map for generalized Bratteli diagrams; see \cite{HermanPutnamSkau1992, Durand2010, BezuglyiKarpel2016, BezuglyiJorgensenKarpelSanadhya2026}.

\begin{definition}\label{Def:VershikMap}  For an ordered generalized 
Bratteli diagram $B=(V,E, >)$, we define a Borel transformation 
\begin{equation}\label{eq: Vershik map}
\varphi_B : X_B \setminus X_{max} \rightarrow X_B \setminus X_{min}
\end{equation}
by the following rule. Given $x = (x_0, x_1,...)\in X_B\setminus X_{max}$, 
let $m$ be the smallest number such that $x_m$ is not maximal. Let 
$g_m$ be the successor of the edge $x_m$ in the finite set $r^{-1}(r(x_m))$.
Then we set $\varphi_B(x)= (g_0, g_1,...,g_{m-1},g_m,x_{m+1},...)$
where $(g_0, g_1,..., g_{m-1})$ is the minimal path in $E(V_0, 
s(g_{m}))$. The map $\varphi_B$ is a Borel bijection. Moreover, $\varphi_B$ is a homeomorphism from 
$X_B\setminus X_{max}$ onto $X_B\setminus X_{min}$. If  
$\varphi_B$ admits a bijective Borel extension to the 
entire path space $X_B$, then we call the Borel transformation 
$\varphi_B : X_B  \rightarrow X_B$ a \textit{Vershik map}, and 
the Borel dynamical system $(X_B,\varphi_B)$ a (generalized) 
\textit{Bratteli-Vershik} system.
\end{definition}

\begin{remark} Of course, condition $|X_{max}| = |X_{min}|$ is necessary and sufficient for the existence of a Borel Vershik map. 
It is natural to ask whether the transformation $\varphi_B$ can be a homeomorphism of the path space $X_B$ endowed with its zero-dimensional Polish topology. An affirmative answer depends on the possibility of pairing maximal and minimal infinite paths so that the extension of $\varphi_B$ is continuous; see \cite{HermanPutnamSkau1992} for more details.  
The question of whether this extension is a homeomorphism for nonsimple standard or generalized Bratteli diagrams is more delicate; see 
\cite{Medynets_2006, BezuglyiYassawi2017, BezuglyiKwiatkowskiYassawi2014, BezuglyiJorgensenKarpelSanadhya2026}.
\end{remark} 

With every Bratteli diagram $B$, we can associate a refining sequence of cylinder partitions (analogous to Kakutani-Rokhlin partitions) formed by the ``towers'' $X_w^{(n)}$ where $w\in V_n$ and $n \in \N_0$ defined by
$$
X_w^{(n)} := \{ x =(x_i) \in X_B : s(x_n) = w\}.
$$
Each finite path $\ov e = (e_0, \ldots, e_{n-1})$ with
$r(e_{n-1}) = w$ determines a ``floor'' of the tower $X_w^{(n)}$. It is identified with the set 
$$
X_w^{(n)}(\ov e) = \{x = (x_i)\in X_B : x_i = e_i,\; i = 
0,\ldots, n-1 \}.
$$
Then 
\be\label{eq X_w^{(n)}} 
X_w^{(n)} = \bigcup_{\ol e \in E(V_0, w)} X_w^{(n)}(\ov e).
\ee

Let $n \in \N$. For $v \in V_n$ and $v_0 \in
V_0$, we set $h^{(n)}_{v, v_0} = |E(v_0, v)|$ and define 
\be\label{eq H^{(n)}_v}
H^{(n)}_v = \sum_{v_0 \in V_0} h^{(n)}_{v, v_0} = |E(v, V_0)|, \ \ n \in \N,
\ee
where $H^{(0)}_v = 1$, $v\in V_0$. Then $h_{v,w}^{(n)}$ is the $(v,w)$-entry in the product of matrices $F_{n-1} \ \cdots \ F_0$, 
$H^{(n)}_v$ is the sum of all entries in the $v$-th row of the matrix $F_{n-1} \ \cdots \ F_0$, and 
 \be\label{eq2: FH=H}
F_n H^{(n)} = H^{(n+1)}, \quad n \in \N_0.
 \ee

\vskip 4mm 

In this paper, we are interested in \textit{positive tail invariant Borel measures} on the path space $X_B$ with full support; that is, every cylinder set has positive measure.
 We refer to \cite{BezuglyiKwiatkowskiMedynetsSolomyak2010, BezuglyiKarpel2016, BezuglyiJorgensenKarpelSanadhya2026} for more information on tail invariant measures.

\begin{definition}\label{def: tail inv meas} Let 
$B =(V, E)$ be 
a Bratteli diagram and $\mathcal R$ the tail equivalence relation on the path space $X_B$. A measure $\mu$ on $X_B$ is called \textit{tail invariant} if, for any cylinder sets $[\ol e]$ and $[\ol e']$ determined by finite paths beginning at level $0$ and satisfying $r(\ol e) = r(\ol e')$, we have $\mu([\ol e]) = \mu([\ol e'])$.

\end{definition}

Let $\mu$ be a tail invariant measure. It follows from \eqref{eq X_w^{(n)}}, \eqref{eq H^{(n)}_v}, and \eqref{eq2: FH=H} that for every path $\ol e$ ending in $w \in V_n$
\be\label{eq def p^(n)}
p^{(n)}_w = \frac{\mu(X_w^{(n)}) }{H^{(n)}_w} = \mu([\ol e]), 
\ee
and this quantity does not depend on $\ol e$. 
We note that the vectors $p^{(n)} = (p^{(n)}_w : w \in V_n)$ determine the tail invariant measure $\mu$ uniquely. 

The following result gives a criterion for the existence of tail-invariant measures. 

\begin{theorem}\label{thm inv measures} 
Let $B =(V, E)$ be a generalized Bratteli diagram defined
by a sequence of incidence matrices $F_n$. Let $\mu$ be a tail invariant Borel probability measure on the path space $X_B$ of $B$.
 Then the corresponding sequence of vectors 
$(p^{(n)})$ (defined as in \eqref{eq def p^(n)}) satisfies the 
property 
\be\label{eq_inv meas via A_n}
A_n p^{(n+1)} = p^{(n)}, 
\ee
where $A_n = F_n^T$ is the transpose of $F_n$.

Conversely, if a sequence of nonnegative vectors $(p^{(n)})$ satisfies 
\eqref{eq_inv meas via A_n}, then it defines  a unique  tail invariant
 measure $\mu$.
 
The theorem remains true for $\sigma$-finite measures $\nu$ 
 satisfying the property that
  $\nu([\ol e]) < \infty$ for every cylinder set $[\ol e]$.
\end{theorem}

We will need the following fact, which follows from \eqref{eq def p^(n)}: if $\mu$ is a tail invariant probability measure, then
\be\label{eq meas of towers}
\sum_{v \in V_n} H^{(n)}_v p_v^{(n)} =1, \quad\ \forall n \in \N_0.  
\ee

\section{Ratio Tests for finite paths}\label{sect Ratio test}

In this section, we discuss the following problem. Given a generalized Bratteli diagram $B = B(F_n)$, find sufficient conditions on incidence matrices $F_n$ under which there is no probability tail invariant measure on the path space $X_B$. We will use the notation introduced in Section \ref{sect BD}. 

We will give two proofs of the following theorem.

\begin{theorem}\label{thm test 1} Let $B = B(F_n)$ be a generalized Bratteli diagram defined by a sequence of incidence matrices $(F_n)$.
Suppose that for all $w\in V_{0}$,
\begin{align}\label{eq suff cond 1}
\liminf_{n\to\infty}\left(\sup_{v\in V_{n}}\frac{h_{v,w}^{(n)}}{H_{v}^{(n)}}\right)=0
\end{align}
Then $B$ does not admit a probability tail invariant measure. Moreover, 
this condition is sufficient but not necessary. 
\end{theorem}

\begin{proof}
Suppose that there exists a tail invariant measure $\mu$ such that $\mu(X_{B})=1$. Then, for all $n\in\mathbb{N}$,
\begin{align*}
\sum_{v\in V_{n}}H_{v}^{(n)}p_{v}^{(n)}=1.
\end{align*}
Fix $w\in V_{0}$ and denote  
\begin{align*}
[w]=\{x\in X_{B}:s(x)=w\},
\end{align*}
where $s(x)=s(x_{0})$. Then, for every $n \in \N$,
\begin{align}\label{eq [w]}
[w]=\bigsqcup_{v\in V_{n}} \ \bigsqcup_{\substack{\overline{e},s(\overline{e})= w\\ 
r(\overline{e})=v}}[\overline{e}].
\end{align}
Indeed, $x\in[w]$ if and only if $x\in E(w,r(x_{n-1}))$. 
This means that the set $[w]$ is represented as the disjoint union of cylinder sets $[\ol e]$ defined by finite paths of length $n$ starting at $w$ and ending at vertices of $V_n$. Then, using the notation defined in Section  \ref{sect BD} and \eqref{eq meas of towers}, \eqref{eq [w]}, we can write
\be \label{eq meas [w]}
\ba
\mu([w])&=\sum_{v\in V_{n}}h_{v,w}^{(n)}p_{v}^{(n)}\\
&=\sum_{v\in V_{n}}h_{v,w}^{(n)}\frac{1}{H_{v}^{(n)}}H_{v}^{(n)}p_{v}^{(n)}\\
&\leq\left(\sup_{v\in V_{n}}\frac{h_{v,w}^{(n)}}{H_{v}^{(n)}}\right)\sum_{v\in V_{n}}H_{v}^{(n)}p_{v}^{(n)}\\
&=\sup_{v\in V_{n}}\frac{h_{v,w}^{(n)}}{H_{v}^{(n)}}.
\ea
\ee
Taking the limit in \eqref{eq meas [w]}, we obtain the inequality 
\begin{align}
\mu([w])\leq\liminf_{n\to\infty}\left(\sup_{v\in V_{n}}\frac{h_{v,w}^{(n)}}{H_{v}^{(n)}}\right)=0.
\end{align}
Hence $\mu([w])=0$ and the relation
\begin{align*}
\mu(X_{B})=\mu\left(\bigsqcup_{w\in V_{0}}[w]\right)=\sum_{w\in V_{0}}\mu([w])=0
\end{align*}
contradicts the assumption about the measure $\mu$.

The claim that \eqref{eq suff cond 1} is not necessary follows from Example \ref{ex not nessas}, which is based, in turn, on Theorem \ref{Thm:measures_Exk}. 
\end{proof}

The result of Theorem \ref{thm test 1} also has a dynamical interpretation in terms of the Vershik map, which will be illustrated by the following alternative proof.

\textit{Proof} (of Theorem \ref{thm test 1})
Assume, toward a contradiction, that both 
$$
\sup_{v\in V_{n}}\frac{h_{v,w}^{(n)}}{H_{v}^{(n)}}\to0
$$ 
as $n\to\infty$, and also that we have a tail invariant probability measure $\mu$ on $X_B$. Moreover, we can assume, without loss of generality, that there is a Vershik map $\varphi$ on the path space $X_{B}$ for some order $B$. This is because the Vershik map is naturally defined everywhere except for the set $X_{B}(\text{max})$ of maximal paths, and $\varphi(X_{B}(\text{max}))\cap X_{B}(\text{max})=\emptyset$ and $\mu(X_{B}(\text{max}))=0$ for any probability tail invariant measure $\mu$. Also, we can assume that the measure $\mu$ is ergodic. More discussion on relations between tail equivalence relations and Vershik maps can be found in \cite{GiordanoPutnamSkau1995, GiordanoPutnamSkau2004, BezuglyiKwiatkowskiMedynetsSolomyak2010}.

For $[w] :=\{x \in X_{B}:s(x)=w\}$, we have
$$
[w]=\bigcup_{v\in V_{n}}Y_{w,v}^{(n)},
$$
where the sets
$$
Y_{v,w}^{(n)} =\{x =(x_{i})\in X_{B}:s(x_{0})=w, r(x_{n-1})=v\}
$$
are disjoint. 
Note also that $X_{v}^{(n)}=\bigcup_{w\in V_{0}}Y_{v,w}^{(n)}$ and that this is a finite disjoint union.

So we compute
$$\mu([w])=\sum_{v\in V_{n}}\mu(Y_{v,w}^{(n)})=\sum_{v\in V_{n}}p^{(n)}(F_{n-1}\cdots F_{0})_{v,w}$$
where $p_{v}^{(n)}=\mu([\overline{e}])$, $r(\overline{e})=v\in V_{n}$. Fix an $\overline{e}\in E(v,w)$ and consider $\chi_{[\overline{e}]}$. Suppose $\overline{x}\in[\overline{e}]$, then $\overline{x}=(x_{n})$ and let $u_{n}=s(x_{n})\in V_{n}$. By the ergodic theorem,
$$\mu([\overline{e}])=\lim_{N\to\infty}\frac{1}{N}\sum_{k=0}^{N-1}\chi_{[\overline{e}]}(\varphi^{k}(\overline{x})),$$
where $\varphi$ is the Vershik map. Since the limit exists and does not depend on $\overline{x}$, we can take a subsequence $\left\{H_{u_n}^{(n)}\right\}$ defined by a fixed path $\overline{x}$. Thus,
$$\mu([\overline{e}])=\lim_{m\to\infty}\frac{1}{H_{u_m}^{(m)}}\sum_{k=0}^{H_{u_m}^{(m)}-1}\chi_{[\overline{e}]}(\varphi^{k}\overline{x}).
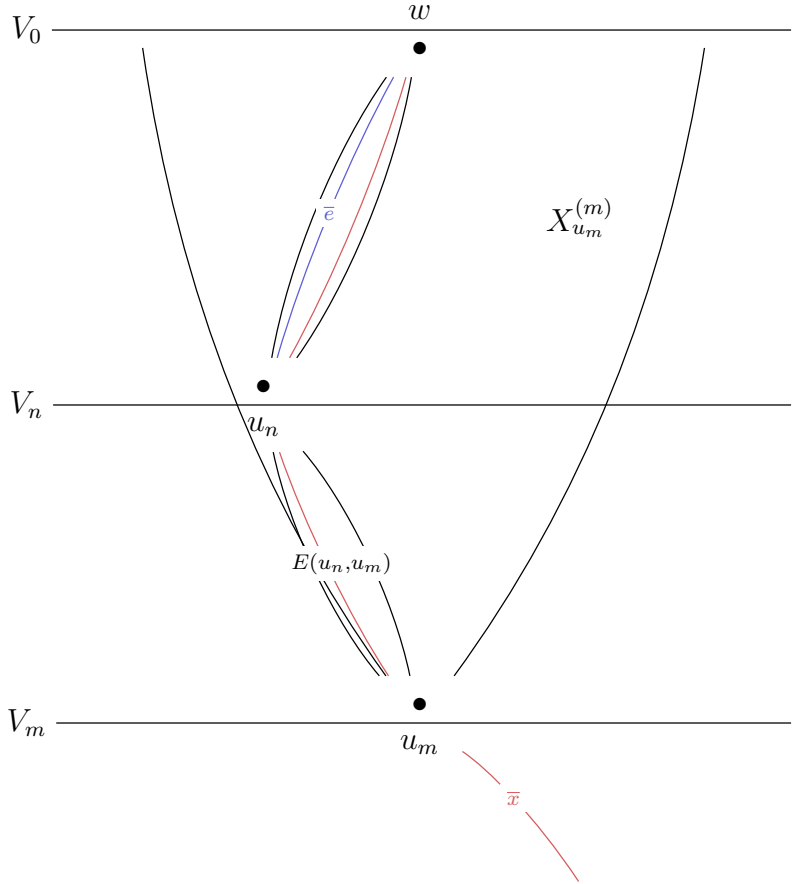
$$

\[\begin{tikzcd}[cramped,row sep=huge]
	{V_0} & {} && \begin{array}{c} w\\\bullet \end{array} && {} & {} \\
	&&&& {X^{(m)}_{u_{m}}} \\
	{V_n} && \begin{array}{c} \bullet\\u_n \end{array} &&&& {} \\
	\\
	{V_m} &&& \begin{array}{c} \bullet\\u_m \end{array} &&& {} \\
	&&&& {}
	\arrow[no head, from=1-1, to=1-7]
	\arrow[curve={height=24pt}, no head, from=1-2, to=5-4]
	\arrow[curve={height=-12pt}, no head, from=1-4, to=3-3]
	\arrow[curve={height=12pt}, no head, from=1-4, to=3-3]
	\arrow["{\overline{e}}"{description}, color={rgb,255:red,92;green,92;blue,214}, curve={height=6pt}, no head, from=1-4, to=3-3]
	\arrow[color={rgb,255:red,214;green,92;blue,92}, curve={height=-6pt}, no head, from=1-4, to=3-3]
	\arrow[curve={height=-24pt}, no head, from=1-6, to=5-4]
	\arrow[no head, from=3-1, to=3-7]
	\arrow[curve={height=12pt}, no head, from=3-3, to=5-4]
	\arrow[curve={height=-12pt}, no head, from=3-3, to=5-4]
	\arrow[color={rgb,255:red,214;green,92;blue,92}, curve={height=6pt}, no head, from=3-3, to=5-4]
	\arrow["{E(u_{n},u_{m})}"{description}, draw=none, from=3-3, to=5-4]
	\arrow[no head, from=5-1, to=5-7]
	\arrow["{\overline{x}}"{description}, color={rgb,255:red,214;green,92;blue,92}, curve={height=-6pt}, no head, from=5-4, to=6-5]
\end{tikzcd}\]

We now claim that the orbit $\{\varphi^{k}(\overline{x})\}$ visits $[\overline{e}]$ exactly $|E(u_{n},u_{m})|$ times. (See the attached diagram). Hence,
\be
\ba
\mu([\overline{e}])&=\lim_{m\to\infty}\frac{1}{H_{u_{m}}^{(m)}}|E(u_{n},u_{m})|\\
&=\lim_{m\to\infty}\frac{1}{H_{u_{m}}^{(m)}}(F_{m-1}\cdots F_{n})_{u_{m},u_{n}}.
\ea
\ee
Then
\be
\ba
\mu(Y_{u_{n},w}^{(n)})&=\mu([\overline{e}])|E(u_{n},w)|\\
&=\mu([\overline{e}])(F_{n-1}\cdots F_{0})_{w,u_{n}}\\
&=\lim_{m\to\infty}\frac{1}{H_{u_{m}}^{(m)}}(F_{m-1}\cdots F_{n})_{u_{m},u_{n}}(F_{n-1}\cdots F_{0})_{w,u_{n}}
\ea
\ee
Note that 
\be
(F_{m-1}\cdots F_{n})_{u_{n},u_{m}}(F_{n-1}\cdots F_{0})_{w,u_{n}} \leq (F_{m-1}\cdots F_{0})_{u_{m},w}\\
 =h_{u_{m},w}^{(m)}.
\ee
Thus,
\be
\ba
\mu(Y_{u_{n},w}^{(n)})&\leq\lim_{m\to\infty}\frac{h_{u_{m},w}^{(m)}}{H_{u_{m}}^{(m)}}\\
&\leq\lim_{m\to\infty}\sup_{v\in V_{m}}\frac{h_{v,w}^{(m)}}{H_{v}^{(m)}}=0.
\ea
\ee
Clearly, if $u\in V_{n}$ is such that $Y_{u,w}^{(n)}\neq\emptyset$, then we can take a path $\overline{x}$ that goes through the vertices $w$ and $u$ and apply the above argument. Thus, we proved that 
$$\mu([w])=\sum_{u\in V_{n}}\mu(Y_{u,w}^{(n)})=0,
$$
which is a contradiction.
\hfill$\Box$

\begin{example}
Let $V_{n}=\mathbb{Z}$ and consider the tridiagonal symmetric matrix $F$ where $f_{ii}=2$ and $f_{i,i\pm1}=1$, and $f_{i,j}=0$ otherwise. Consider the Bratteli diagram $B$ where $B=B(F)$.

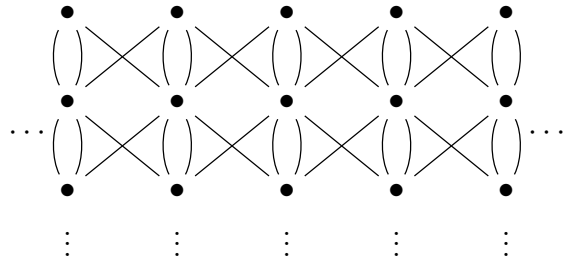
\[\cdots\begin{tikzcd}[cramped]
	\bullet & \bullet & \bullet & \bullet & \bullet \\
	\bullet & \bullet & \bullet & \bullet & \bullet \\
	\bullet & \bullet & \bullet & \bullet & \bullet \\
	{} & {} & {} & {} & {}
	\arrow[curve={height=-6pt}, no head, from=1-1, to=2-1]
	\arrow[curve={height=6pt}, no head, from=1-1, to=2-1]
	\arrow[no head, from=1-1, to=2-2]
	\arrow[no head, from=1-2, to=2-1]
	\arrow[curve={height=-6pt}, no head, from=1-2, to=2-2]
	\arrow[curve={height=6pt}, no head, from=1-2, to=2-2]
	\arrow[no head, from=1-2, to=2-3]
	\arrow[no head, from=1-3, to=2-2]
	\arrow[curve={height=-6pt}, no head, from=1-3, to=2-3]
	\arrow[curve={height=6pt}, no head, from=1-3, to=2-3]
	\arrow[no head, from=1-3, to=2-4]
	\arrow[no head, from=1-4, to=2-3]
	\arrow[curve={height=-6pt}, no head, from=1-4, to=2-4]
	\arrow[curve={height=6pt}, no head, from=1-4, to=2-4]
	\arrow[no head, from=1-4, to=2-5]
	\arrow[no head, from=1-5, to=2-4]
	\arrow[curve={height=-6pt}, no head, from=1-5, to=2-5]
	\arrow[curve={height=6pt}, no head, from=1-5, to=2-5]
	\arrow[curve={height=6pt}, no head, from=2-1, to=3-1]
	\arrow[curve={height=-6pt}, no head, from=2-1, to=3-1]
	\arrow[no head, from=2-1, to=3-2]
	\arrow[no head, from=2-2, to=3-1]
	\arrow[curve={height=6pt}, no head, from=2-2, to=3-2]
	\arrow[curve={height=-6pt}, no head, from=2-2, to=3-2]
	\arrow[no head, from=2-2, to=3-3]
	\arrow[no head, from=2-3, to=3-2]
	\arrow[curve={height=-6pt}, no head, from=2-3, to=3-3]
	\arrow[curve={height=6pt}, no head, from=2-3, to=3-3]
	\arrow[no head, from=2-3, to=3-4]
	\arrow[no head, from=2-4, to=3-3]
	\arrow[curve={height=6pt}, no head, from=2-4, to=3-4]
	\arrow[curve={height=-6pt}, no head, from=2-4, to=3-4]
	\arrow[no head, from=2-4, to=3-5]
	\arrow[no head, from=2-5, to=3-4]
	\arrow[curve={height=6pt}, no head, from=2-5, to=3-5]
	\arrow[curve={height=-6pt}, no head, from=2-5, to=3-5]
	\arrow["\vdots"{description}, draw=none, from=3-1, to=4-1]
	\arrow["\vdots"{description}, draw=none, from=3-2, to=4-2]
	\arrow["\vdots"{description}, draw=none, from=3-3, to=4-3]
	\arrow["\vdots"{description}, draw=none, from=3-4, to=4-4]
	\arrow["\vdots"{description}, draw=none, from=3-5, to=4-5]
\end{tikzcd}\cdots\]

It is clear that we have the recurrence relations $h_{v,w}^{(n)}=2h_{v,w}^{(n-1)}+h_{v+1,w}^{(n-1)}+h_{v-1,w}^{(n-1)}$. We observe that this results in us having every other level of Pascal's triangle, as in the following, where the middle represents $h^{(n)}_{w,w}$:
$$
\begin{matrix}
w-v:&4&3&2&1&0&-1&-2&-3&-4\\
\hline
h^{(0)}_{v,w}:&&&&&1\\
h^{(1)}_{v,w}:&&&&1&2&1\\
h^{(2)}_{v,w}:&&&1&4&6&4&1\\
h^{(3)}_{v,w}:&&1&6&15&20&15&6&1\\
h^{(4)}_{v,w}:&1&8&28&56&70&56&28&8&1\\
\vdots&&&&\vdots
\end{matrix}
$$
Note that the choice of $w$ may be arbitrary, as $B$ is horizontally stationary.
 Therefore, it follows that, for $|w-v|\leq n$,
$$h^{(n)}_{v,w}=\begin{pmatrix}2n\\n-(v-w)\end{pmatrix},$$
which is the binomial coefficient, and
$$H^{(n)}_{v}=4^{n}.$$ Since clearly $h^{(n)}_{v,w}$ attains its largest value when $v=w$, we compute that
$$
\liminf_{n\to\infty}\sup_{v\in V_{n}}\frac{h^{(n)}_{v,w}}{H^{(n)}_{v}}=\lim_{n\to\infty}\frac{(2n)!}{(n!)^{2}\cdot 4^{n}}=0.
$$
Therefore, by Theorem \ref{thm test 1}, this Bratteli diagram admits no tail invariant probability measure.
\end{example}

\begin{remark}
This example was considered in \cite{BezuglyiJorgensenKarpelKwiatkowski2025}. The advantage of the method used in Theorem \ref{thm test 1} is its computational ease. We expect that other results from \cite{BezuglyiJorgensenKarpelKwiatkowski2025} about the Bratteli diagrams that do not support probability tail invariant measures can be obtained by applying Theorem \ref{thm test 1}. 

\end{remark}

We show now that condition \eqref{eq suff cond 1} is sufficient but not necessary; see the example below. For this, we first recall the following result proved in \cite{BezuglyiKarpelKwiatkowski2024}. 

\begin{theorem}\label{Thm:measures_Exk}
    Let $B$ be a stationary generalized Bratteli diagram
 with incidence matrix
 $$
F = \begin{pmatrix}
    a & 1 & 0 & 0 & 0 &\ldots\\
    0 & a-k & 1 & 0 & 0 &\ldots\\
    0 & 0 & a-k & 1 & 0 &\ldots\\
    0 & 0 & 0 & a-k & 1 &\ldots\\
    \vdots & \vdots & \vdots & \vdots & \vdots &\ddots
\end{pmatrix},
 $$
 where $a, k \in \mathbb{N}$ and $a - k > 1$. 
 Then there is a unique probability ergodic invariant measure $\mu$ on $B$ if and only if $k > 1$. 
 If $k = 1$, then there are no probability invariant measures on $B$. 
 \end{theorem}

Theorem \ref{Thm:measures_Exk} is used in the following example.

\begin{example} \label{ex not nessas}
Let $B = B(F)$ be the stationary generalized Bratteli diagram defined by the incidence matrix 
$$
F=\begin{pmatrix}
3 &1&0&0&\cdots\\0&2&1&0&\cdots\\0&0&2&1&\cdots\\\vdots&\vdots&\vdots&\ddots&\ddots
\end{pmatrix}
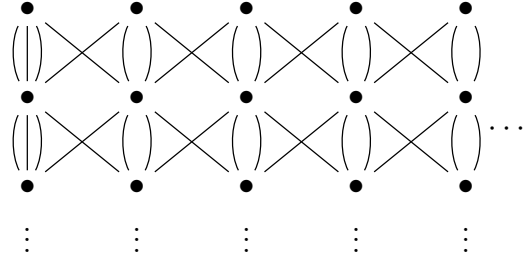
$$

\[\begin{tikzcd}[cramped]
	\bullet & \bullet & \bullet & \bullet & \bullet \\
	\bullet & \bullet & \bullet & \bullet & \bullet \\
	\bullet & \bullet & \bullet & \bullet & \bullet \\
	{} & {} & {} & {} & {}
	\arrow[curve={height=-6pt}, no head, from=1-1, to=2-1]
	\arrow[curve={height=6pt}, no head, from=1-1, to=2-1]
	\arrow[no head, from=1-1, to=2-1]
	\arrow[no head, from=1-2, to=2-1]
	\arrow[curve={height=-6pt}, no head, from=1-2, to=2-2]
	\arrow[curve={height=6pt}, no head, from=1-2, to=2-2]
	\arrow[no head, from=1-3, to=2-2]
	\arrow[curve={height=-6pt}, no head, from=1-3, to=2-3]
	\arrow[curve={height=6pt}, no head, from=1-3, to=2-3]
	\arrow[no head, from=1-4, to=2-3]
	\arrow[curve={height=-6pt}, no head, from=1-4, to=2-4]
	\arrow[curve={height=6pt}, no head, from=1-4, to=2-4]
	\arrow[no head, from=1-5, to=2-4]
	\arrow[curve={height=-6pt}, no head, from=1-5, to=2-5]
	\arrow[curve={height=6pt}, no head, from=1-5, to=2-5]
	\arrow[curve={height=6pt}, no head, from=2-1, to=3-1]
	\arrow[curve={height=-6pt}, no head, from=2-1, to=3-1]
	\arrow[no head, from=2-1, to=3-1]
	\arrow[no head, from=2-2, to=3-1]
	\arrow[curve={height=6pt}, no head, from=2-2, to=3-2]
	\arrow[curve={height=-6pt}, no head, from=2-2, to=3-2]
	\arrow[no head, from=2-3, to=3-2]
	\arrow[curve={height=-6pt}, no head, from=2-3, to=3-3]
	\arrow[curve={height=6pt}, no head, from=2-3, to=3-3]
	\arrow[no head, from=2-4, to=3-3]
	\arrow[curve={height=6pt}, no head, from=2-4, to=3-4]
	\arrow[curve={height=-6pt}, no head, from=2-4, to=3-4]
	\arrow[no head, from=2-5, to=3-4]
	\arrow[curve={height=6pt}, no head, from=2-5, to=3-5]
	\arrow[curve={height=-6pt}, no head, from=2-5, to=3-5]
	\arrow["\vdots"{description}, draw=none, from=3-1, to=4-1]
	\arrow["\vdots"{description}, draw=none, from=3-2, to=4-2]
	\arrow["\vdots"{description}, draw=none, from=3-3, to=4-3]
	\arrow["\vdots"{description}, draw=none, from=3-4, to=4-4]
	\arrow["\vdots"{description}, draw=none, from=3-5, to=4-5]
\end{tikzcd}\]

Then by Theorem \ref{Thm:measures_Exk}, we know that there is no probability tail invariant measure on $X_B$. We will show that condition 
\eqref{eq suff cond 1} does not hold. For this, it suffices to find a vertex $w_0 \in V_0$ such that, for some $\delta >0$,
$$
\sup_{v\in V_{n}}\frac{h_{v,w_0}^{(n)}}{H_{v}^{(n)}} > \delta.
$$
Clearly, if $v > 1$, then $H^{(n)}_{v}=3^{n}$ and $h^{(n)}_{v,w} < 2^n$, so the ratio $\frac{h^{(n)}_{v,w}}{H^{(n)}_{v}}$ tends to 0 as $n \to \infty$. On the other hand, if $w_0 =1$, then 
$$
\frac{h^{(n)}_{1,1}}{H^{(n)}_{1}} = \frac{3^n}{3^n + 3^{n-1}} = \frac{3}{4}.
$$
This shows that condition \eqref{eq suff cond 1} is not necessary.

\end{example}

The following result can be considered as a generalization of Theorem \ref{thm test 1}.

\begin{theorem} Let $B = B(F_n)$ be a generalized Bratteli diagram defined by a sequence of incidence matrices $(F_n)$.
Suppose that there exists some level $k$ such that, for all finite subsets $W\subset V_{k}$, we have  
\be\label{eq suff cond 2}
\liminf_{n\to\infty} \left( \sup_{v\in V_{n}}\frac{\sum_{w\in W}H_{w}^{(k)}h_{v,w}^{(n,k)}}{H_{v}^{(n)}}\right) =0,
\ee
where $h_{v,w}^{(n,k)}$ denotes the number of finite paths between the vertices $w \in W$ and $v \in V_n, n >k.$
Then the generalized Bratteli diagram $B$ admits no probability tail invariant measure.
\end{theorem}

\begin{proof}
We recall that if $\mu$ is a probability tail invariant measure, then for all $n$,
$$
\mu(X_{B})=\sum_{v\in V_{n}}H_{v}^{(n)}p_{v}^{(n)}=1.
$$
Let $W\subset V_{k}$ be a finite set. Denote by $[W]_{n}$ the disjoint union of cylinder sets of length $n$ whose $k$-th vertex belongs to $W$.
That is,
$$
[W]_{n}=\bigsqcup_{w\in W}\bigsqcup_{v_{0}\in V_{0}}\bigsqcup_{v\in V_{n}}\bigsqcup_{\substack{\overline{e}, s(\overline{e})=v_{0} \\ r(\overline{e})=v, r(e_{k-1}) =w}}[\overline{e}].
$$
Note that $[W]_{n}$ is not defined for $n<k$.

The quantity $H_{w}^{(k)}h_{v,w}^{(n,k)}$ gives the total number of all finite paths passing through $w$ and ending at $v$ at step $n$. Then
\be
\ba
\mu([W]_{n})&=\sum_{v\in V_{n}}\left(\sum_{w\in W}H_{w}^{(k)}h_{v,w}^{(n,k)}\right)p_{v}^{(n)}\\
&=\sum_{v\in V_{n}}\left(\frac{\sum_{w\in W}H_{w}^{(k)}h_{v,w}^{(n,k)}}{H_{v}^{(n)}}\right)H_{v}^{(n)}p_{v}^{(n)}\\
&\leq\sup_{v\in V_{n}}\left(\frac{\sum_{w\in W}H_{w}^{(k)}h_{v,w}^{(n,k)}}{H_{v}^{(n)}}\right)\sum_{v\in V_{n}}H_{v}^{(n)}p_{v}^{(n)}\\
&=\sup_{v\in V_{n}}\frac{\sum_{w\in W}H_{w}^{(k)}h_{v,w}^{(n,k)}}{H_{v}^{(n)}}.
\ea
\ee

As in the proof of Theorem \ref{thm test 1}, we note that, in fact, $[W]_n$ does not depend on $n$, so that $[W]=[W]_{n}$ and
$$
\mu([W])\leq\liminf_{n\to\infty}\left( \sup_{v\in V_{n}}\frac{\sum_{w\in W}H_{w}^{(k)}h_{v,w}^{(n,k)}}{H_{v}^{(n)}}\right)
$$
for all $n$. Finally, we apply \eqref{eq suff cond 2} to get 
$$\mu(X_{B})=\mu\left(\bigcup_{\substack{W\subseteq V_{k}\\ |W|<\infty}}[W]\right)\leq\sum_{\substack{W\subseteq V_{k}\\ |W|<\infty}}\mu([W])=0,
$$
which is a contradiction.
\end{proof}

There is a simple sufficient condition for the nonexistence of probability tail invariant measures that follows from the proved results.

\begin{theorem}\label{thm 3} Let $B = B(F_n)$ be a generalized Bratteli diagram defined by a sequence of incidence matrices $(F_n)$.
Suppose that there exists a sequence of positive numbers $(M_{n})$ with $M_{n}\to\infty$ such that for all $n\geq1$ and vertices $v\in V_{n}$ and $w\in V_{0}$,
\begin{align}
h_{v,w}^{(n)}M_{n}\leq H_{v}^{(n)}.\label{no more than 1/M}
\end{align}
Then the generalized Bratteli diagram $B$ does not admit any probability tail invariant measure.
\end{theorem}

\begin{proof}
The proof immediately follows from Theorem \ref{thm test 1}. The hypothesis of the theorem means that 
$$\frac{h_{v,w}^{(n)}}{H_{v}^{(n)}}\leq\frac{1}{M_{n}}\to 0, \quad n\to \infty, 
$$ for all $v\in V_{n}$ and $w\in V_{0}$. Then 
$$\lim_{n\to \infty} \sup_{v\in V_{n}}\frac{h_{v,w}^{(n)}}{H_{v}^{(n)}} =0.$$
\end{proof}

\begin{remark}
\begin{enumerate}[label=(\alph*)]
\item The ratio $\frac{h_{v,w}^{(n)}}{H_{v}^{(n)}}$ is the fraction of paths ending at $v$ and starting at $w$. Relation \eqref{no more than 1/M} means that no single vertex $w$ contributes more than $M_{n}^{-1}$. This condition can be checked for various types of generalized Bratteli 
diagrams. In particular, for some classes of stationary Bratteli diagrams where the entries of $F^n$ grow as powers $\lambda^n$ of the Perron eigenvalue $\lambda$.

\item Suppose that $U_{n}(v)=\{w\in V_{0}:h_{v,w}^{(n)}>0\}$ is the set of vertices from $V_0$ that are connected by a path with a vertex $v\in V_n$.
This set is finite; see Definition \ref{def BD}. Assume that 
there exists a sequence $M_{n}\to\infty$ such that $|U_{n}(v)|\geq M_{n}$. This occurs if $(F_{n})$ are banded matrices. The average value of $h_{v,w}^{(n)}$ over $w\in U_{n}(v)$ is $\dfrac{H_{v}^{(n)}}{|U_{n}(v)|}$. If the paths are \textit{balanced}, that is, 
$$h_{v,w}^{(n)}\leq C\cdot\frac{H_{v}^{(n)}}{|U_{n}(v)|},
$$ then 
$$\frac{h_{v,w}^{(n)}}{H_{v}^{(n)}}\leq\frac{C}{M_{n}}\to 0,
$$ and we can apply Theorem \ref{thm test 1}.
\end{enumerate}
\end{remark}

The following example illustrates Theorem \ref{thm 3}.

\begin{example}\label{ex triangle}
Let $V_{n}=\mathbb{Z}$ for all $n\geq0$. Construct a stationary generalized Bratteli diagram $B= B(F)$, and let $A= F^T$ have the entries
$a_{i,j}$ where $a_{i,i}=a_{i,i+1}=1$ and $a_{i,j}=0$ otherwise. We compute $h_{v,w}^{(n)}=\begin{pmatrix}n\\v-w\end{pmatrix}$ for $0\leq v-w\leq n$. Then 
$$H_{v}^{(n)}=\sum_{w=v-n}^{v}\begin{pmatrix}n\\v-w\end{pmatrix}=\sum_{k=0}^{n}\begin{pmatrix}n\\k\end{pmatrix}=2^{n}.$$
The maximum value of $\begin{pmatrix}n\\k\end{pmatrix}$ occurs if $k=\lfloor\dfrac{n}{2}\rfloor$. So,
$$h_{v,w}^{(n)}\leq\begin{pmatrix}n\\\lfloor n/2\rfloor\end{pmatrix}\approx\frac{2^{n}}{\sqrt{\pi n/2}}$$
Therefore,
$$\frac{h_{v,w}^{(n)}}{H_{v}^{(n)}}\leq\frac{\begin{pmatrix}n\\\lfloor n/2\rfloor\end{pmatrix}}{2^{n}}<c\frac{1}{\sqrt{n}}=:\frac{1}{M_{n}}\to0$$
as $n\to\infty$. By Theorem \ref{thm 3}, this diagram has no probability invariant measure.

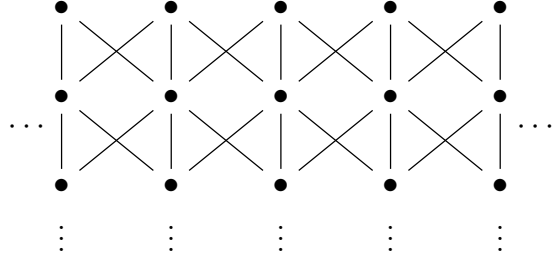
\[\cdots\begin{tikzcd}[cramped]
	\bullet & \bullet & \bullet & \bullet & \bullet \\
	\bullet & \bullet & \bullet & \bullet & \bullet \\
	\bullet & \bullet & \bullet & \bullet & \bullet \\
	{} & {} & {} & {} & {}
	\arrow[no head, from=1-1, to=2-1]
	\arrow[no head, from=1-1, to=2-2]
	\arrow[no head, from=1-2, to=2-1]
	\arrow[no head, from=1-2, to=2-2]
	\arrow[no head, from=1-2, to=2-3]
	\arrow[no head, from=1-3, to=2-2]
	\arrow[no head, from=1-3, to=2-3]
	\arrow[no head, from=1-3, to=2-4]
	\arrow[no head, from=1-4, to=2-3]
	\arrow[no head, from=1-4, to=2-4]
	\arrow[no head, from=1-4, to=2-5]
	\arrow[no head, from=1-5, to=2-4]
	\arrow[no head, from=1-5, to=2-5]
	\arrow[no head, from=2-1, to=3-1]
	\arrow[no head, from=2-1, to=3-2]
	\arrow[no head, from=2-2, to=3-1]
	\arrow[no head, from=2-2, to=3-2]
	\arrow[no head, from=2-2, to=3-3]
	\arrow[no head, from=2-3, to=3-2]
	\arrow[no head, from=2-3, to=3-3]
	\arrow[no head, from=2-3, to=3-4]
	\arrow[no head, from=2-4, to=3-3]
	\arrow[no head, from=2-4, to=3-4]
	\arrow[no head, from=2-4, to=3-5]
	\arrow[no head, from=2-5, to=3-4]
	\arrow[no head, from=2-5, to=3-5]
	\arrow["\vdots"{description}, draw=none, from=3-1, to=4-1]
	\arrow["\vdots"{description}, draw=none, from=3-2, to=4-2]
	\arrow["\vdots"{description}, draw=none, from=3-3, to=4-3]
	\arrow["\vdots"{description}, draw=none, from=3-4, to=4-4]
	\arrow["\vdots"{description}, draw=none, from=3-5, to=4-5]
\end{tikzcd}\cdots\]

\end{example}

In the rest of this section, we consider two more approaches for the problem of the nonexistence of probability tail invariant measures. It will be convenient to interpret tail invariance for a measure as a uniformly distributed probability of getting a vertex $v \in V_n$ starting at $V_0$.

For $v\in V_{n}$, suppose that paths terminated at $v$ have the same probability. Then
$$p^{(n)}(v,w):=\frac{h_{v,w}^{(n)}}{H_{v}^{(n)}}
$$
is the probability that a randomly chosen path that ends at $v$ begins at $w\in V_{0}$. Clearly, $\sum_{w\in V_{0}}p^{(n)}(v,w)=1$. Similarly, if $W\subseteq V_{0}$, then $p^{(n)}(v,W)=\sum_{w\in W}p^{(n)}(v,w).$

\begin{theorem} \label{thm W finite}
For a generalized Bratteli diagram $B$,
suppose that for all finite $W\subseteq V_{0}$ the probability 
$p^{(n)}(v,W)$ of originating from $W$ vanishes uniformly as $n\to\infty$, that is:
$$\liminf_{n\to\infty}\left(\sup_{v\in V_{n}}p^{(n)}(v,W)\right)=0.$$
Then the diagram $B$ has no probability tail invariant measures.
\end{theorem}

\begin{proof} In the proof, we will use the notation introduced in
Theorem \ref{thm test 1}.
Assuming that $\mu$ is probability, we have $\mu(X_{B})=\sum_{v\in V_{n}}H_{v}^{(n)}p_{v}^{(n)}= 1.$ Let $[w]$ be the set of all paths beginning in $w\in V_0$, and $[W] = \bigsqcup_{w\in W}[w]$. As noted in Theorem \ref{thm test 1}, the set $[W] = [W]_n$, where $[W]_n$ is formed by all cylinder sets of length $n$.
 Then
\be
\ba
\mu([W]_{n})&=\sum_{w\in W}\sum_{v\in V_{n}}h_{v,w}^{(n)}p_{v}^{(n)}\\
&=\sum_{v\in V_{n}}\left(\sum_{w\in W}h_{v,w}^{(n)}\right)p_{v}^{(n)}\\
&=\sum_{v\in V_{n}}\left(\sum_{w\in W}\frac{h_{v,w}^{(n)}}{H_{v}^{(n)}}\right)H_{v}^{(n)}p_{v}^{(n)}\\
&=\sum_{v\in V_{n}}p^{(n)}(v,W)H_{v}^{(n)}p_{v}^{(n)}\\
&\leq\left(\sup_{v\in V_{n}}p^{(n)}(v,W)\right)\sum_{v\in V_{n}}H_{v}^{(n)}p_{v}^{(n)}\\
&\leq\sup_{v\in V_{n}}p^{(n)}(v,W).
\ea
\ee
Since this holds for all $n$, we conclude that
$$\mu([W])\leq\liminf_{n\to\infty}(\sup_{v\in V_{n}}p^{(n)}(v,W)).$$
Let $W_{1}\subseteq W_{2}\subseteq\cdots\uparrow V_{0}$ be an increasing sequence of finite sets whose union is $V_{0}$. Then $\mu(X_{B})=\lim_{k\to\infty}\mu([W_{k}])=0$.
\end{proof}

\begin{remark}
It follows from the proof of Theorem \ref{thm W finite} that the result remains true if we take an increasing exhaustion
$$
W_1 \subset W_2\subset W_3 \subset \cdots,\quad \bigcup_{i=1}^\infty W_i = V_0.
$$
\end{remark}

\begin{example} Let $B$ be as in Example \ref{ex triangle}, that is 
$a_{n,n-1}=a_{n,n+1}=1$. If $W=\{-K,\dots,K\}\subseteq V_{0}$, then
$$p^{(n)}(v,W)\approx \frac{|W|}{\sqrt{2\pi n}}\exp\left(-\frac{v^{2}}{2n}\right)\leq\frac{2k+1}{\sqrt{2\pi n}}\to0.$$

This diagram $B$ has no probability tail invariant measure.

\end{example}

\begin{definition}
We say that a generalized Bratteli diagram has \textit{uniform bounded distortion} if there is some $C\geq1$ where, for all $m\geq0$, there is some $N(m)$ such that for all $n\geq N(m)$ and $u,v\in V_{n}$,
\begin{align}\label{eq distortion}
\frac{(F^{(n,m)})_{v,u'}}{(F^{(n,m)})_{u,u'}}\leq C\frac{H_{v}^{(n)}}{H_{u}^{(n)}},\qquad\forall u'\in V_{m}.
\end{align}
Here, $F^{(n,m)}=F_{n-1}\cdots F_{m}$ 
and $(F^{(n,m)})_{v,u'}$ shows the number of paths between $u'\in V_{m}$ and $v\in V_{n}$.
\end{definition}
In particular, if $m=0$ and $w\in V_{0}$, then \eqref{eq distortion} has the form
\begin{align}\label{eq distortion 2}
\frac{h_{v,w}^{(n)}}{h_{u,w}^{(n)}}\leq C\frac{H_{v}^{(n)}}{H_{u}^{(n)}}\iff \frac{h_{v,w}^{(n)}}{H_{v}^{(n)}}\leq C\frac{h_{u,w}^{(n)}}{H_{u}^{(n)}}, \qquad\forall v,u\in V_{n},\forall w\in V_{0}.
\end{align}
The relations \eqref{eq distortion} and \eqref{eq distortion 2} hold for the pairs of vertices that are connected by at least one finite path.

\begin{theorem}
Let $B$ be a Bratteli diagram with uniform bounded distortion. Suppose  that for every $w\in V_{0}$ there exists a sequence $(v_{n})$, $v_{n}\in V_{n}$ such that 
$$\liminf_{n\to\infty}h_{v_{n},w}^{(n)}(H_{v_{n}}^{(n)})^{-1}=0.$$
Then, $B$ has no probability tail invariant measure.
\end{theorem}

\begin{proof}
Fix $w\in V_{0}$. By assumption, for all $n>n(0)$ and $v\in V_{n}$ we can choose $u=v_{n}\in V_{n}$ and apply \eqref{eq distortion 2} Then
$$\frac{h_{v,w}^{(n)}}{H_{v}^{(n)}}\leq C\frac{h_{v_{n},w}^{(n)}}{H_{v_{n}}^{(n)}}$$
So, 
$$0\leq\sup_{v\in V_{n}}\frac{h_{v,w}^{(n)}}{H_{v}^{(n)}}\leq C\frac{h_{v_{n},w}^{(n)}}{H_{v_{n}}^{(n)}}.$$
The right-hand term has a subsequence that goes to 0 as $n\to\infty$, so
$$\liminf_{n\to\infty} \left(\sup_{v\in V_{n}}\frac{h_{v,w}^{(n)}}{H_{v}^{(n)}}\right) =0.$$
Thus, by Theorem \ref{thm test 1}, there is no probability tail invariant measure.
\end{proof}

\section{Spectral Method for Stationary Bratteli Diagrams}\label{sect 4}

In this section, we study a class of stationary Bratteli diagrams $B = B(F)$ that do not admit a probability tail invariant measure. We will begin with a particular example of a diagram 
that contains a countable family of odometer subdiagrams with multiplicities $2, 3, 4, \ldots$ coupled by one edge in each direction between neighboring indices, so 
\be\label{eq m-x F}
F =
\begin{pmatrix}
2&1&0&0&\cdots\\
1&3&1&0&\cdots\\
0&1&4&1&\cdots\\
0&0&1&5&\ddots\\
\vdots&\vdots&\vdots&\ddots&\ddots 
\end{pmatrix}.
\ee
This is a tridiagonal symmetric matrix, that is, $A = A^T$, where $A$ is the transpose of $F$. 
We regard $A$ as an operator acting in $\ell^2(\mathbb N)$ defined by
$$
(Ax)_i=(i+1)x_i +x_{i-1}+ x_{i+1},\qquad i\ge 1,
$$
with the convention $x_0=0$. Then $A$ is an unbounded linear operator 
with domain
$$
D(A)=\left\{
x=(x_i)\in\ell^2(\mathbb N):
\bigl((i+1)x_i\bigr)_{i\ge1}\in\ell^2(\mathbb N)
\right\}.
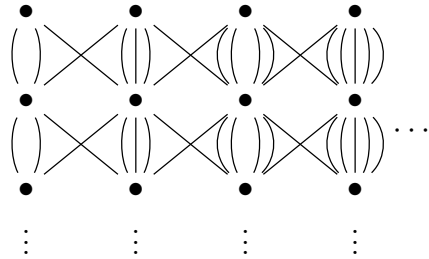
$$

\[\begin{tikzcd}[cramped]
	\bullet & \bullet & \bullet & \bullet \\
	\bullet & \bullet & \bullet & \bullet \\
	\bullet & \bullet & \bullet & \bullet \\
	{} & {} & {} & {}
	\arrow[curve={height=-6pt}, no head, from=1-1, to=2-1]
	\arrow[curve={height=6pt}, no head, from=1-1, to=2-1]
	\arrow[no head, from=1-1, to=2-2]
	\arrow[no head, from=1-2, to=2-1]
	\arrow[curve={height=-6pt}, no head, from=1-2, to=2-2]
	\arrow[curve={height=6pt}, no head, from=1-2, to=2-2]
	\arrow[no head, from=1-2, to=2-2]
	\arrow[no head, from=1-2, to=2-3]
	\arrow[no head, from=1-3, to=2-2]
	\arrow[curve={height=-6pt}, no head, from=1-3, to=2-3]
	\arrow[curve={height=6pt}, no head, from=1-3, to=2-3]
	\arrow[curve={height=12pt}, no head, from=1-3, to=2-3]
	\arrow[curve={height=-12pt}, no head, from=1-3, to=2-3]
	\arrow[no head, from=1-3, to=2-4]
	\arrow[no head, from=1-4, to=2-3]
	\arrow[curve={height=-6pt}, no head, from=1-4, to=2-4]
	\arrow[curve={height=6pt}, no head, from=1-4, to=2-4]
	\arrow[curve={height=-12pt}, no head, from=1-4, to=2-4]
	\arrow[curve={height=12pt}, no head, from=1-4, to=2-4]
	\arrow[no head, from=1-4, to=2-4]
	\arrow[curve={height=6pt}, no head, from=2-1, to=3-1]
	\arrow[curve={height=-6pt}, no head, from=2-1, to=3-1]
	\arrow[no head, from=2-1, to=3-2]
	\arrow[no head, from=2-2, to=3-1]
	\arrow[curve={height=6pt}, no head, from=2-2, to=3-2]
	\arrow[curve={height=-6pt}, no head, from=2-2, to=3-2]
	\arrow[no head, from=2-2, to=3-2]
	\arrow[no head, from=2-2, to=3-3]
	\arrow[no head, from=2-3, to=3-2]
	\arrow[curve={height=-6pt}, no head, from=2-3, to=3-3]
	\arrow[curve={height=6pt}, no head, from=2-3, to=3-3]
	\arrow[curve={height=12pt}, no head, from=2-3, to=3-3]
	\arrow[curve={height=-12pt}, no head, from=2-3, to=3-3]
	\arrow[no head, from=2-3, to=3-4]
	\arrow[no head, from=2-4, to=3-3]
	\arrow[curve={height=6pt}, no head, from=2-4, to=3-4]
	\arrow[curve={height=-6pt}, no head, from=2-4, to=3-4]
	\arrow[curve={height=-12pt}, no head, from=2-4, to=3-4]
	\arrow[no head, from=2-4, to=3-4]
	\arrow[curve={height=12pt}, no head, from=2-4, to=3-4]
	\arrow["\vdots"{description}, draw=none, from=3-1, to=4-1]
	\arrow["\vdots"{description}, draw=none, from=3-2, to=4-2]
	\arrow["\vdots"{description}, draw=none, from=3-3, to=4-3]
	\arrow["\vdots"{description}, draw=none, from=3-4, to=4-4]
\end{tikzcd}\cdots\]

\begin{theorem}\label{thm A s-a}
The operator $A: \ell^2(\mathbb N) \to \ell^2(\mathbb N)$ is self-adjoint, has a compact resolvent, and consequently has a purely discrete spectrum and an orthonormal basis of eigenvectors.
\end{theorem}

\begin{proof}
\textit{Step 1}. Show that $A$ is self-adjoint. Write
$A=D + K$ where $D$ is the diagonal operator $(Dx)_i=(i+1)x_i$
and $(Kx)_i=x_{i-1}+x_{i+1}.$
The diagonal operator $D$ is self-adjoint on the domain 
$D(D)=D(A)$.
The operator $K$ is bounded and self-adjoint on $\ell^2(\mathbb N)$, because $K = S + S^*$, where $S$ is  the unilateral shift on $\ell^2(\mathbb N)$. In fact, $\|K\|\le 2$ since 
$$
\|Kx\| \ \le 2\|x\|.
$$
Applying the Kato-Rellich theorem (the bounded perturbation theorem), we conclude that $A$ is self-adjoint on the maximal domain $D(A)$.
\\

\textit{Step 2}. We  will show that $(A + I)^{-1}$ is bounded. 
For simplicity, we will deal with the real Hilbert space $\ell^2(\mathbb N)$. 
For $x \in D(A)$, compute 
$$
\langle Ax,x\rangle = \sum_{i\ge1}(i+1)|x_i|^2 + \sum_{i\ge 1}x_i x_{i+1} + \sum_{i\ge 1}x_{i-1} {x_{i}} 
= \sum_{i\ge1}(i+1)|x_i|^2 + 2 \sum_{i\ge1}x_i x_{i+1}.
$$
Using
$$
2(x_i x_{i+1}) \ge -|x_i|^2-|x_{i+1}|^2,
$$
we obtain 
\begin{align*}
\langle Ax,x\rangle
&\ge
\sum_{i\ge1}(i+1)|x_i|^2 - \sum_{i\ge1}|x_i|^2 - \sum_{i\ge1}|x_{i+1}|^2\\
&= |x_1|^2+ \sum_{i \ge 2}(i-1) |x_i|^2\\
&\ge \|x\|^2.
\end{align*}
Then
$$
2|| x||^2 \le |\langle (A+I)x, x \rangle| \le ||(A+I)x||\cdot ||x ||
$$
and 
$$
\Vert(A+I)x\Vert \ge 2\Vert x\Vert \quad \text{for all } x \in D(A).
$$
This directly implies that $\ker(A+I) = \{0\}$ (so $A+I$ is injective) and
the range $\operatorname{Ran}(A+I)$ is closed.

Since $A$ (and thus $A+I$) is self-adjoint, we conclude
$$
\operatorname{Ran}(A+I) = \overline{\operatorname{Ran}(A+I)} = \Big(\ker(A+I)^*\Big)^\perp = \Big(\ker(A+I)\Big)^\perp = \{0\}^\perp = \ell^2(\mathbb{N}). 
$$
Now, since $A+I : D(A) \to \ell^2(\mathbb{N})$ is a bijection, its inverse $(A+I)^{-1}$ exists. Setting $y = (A+I)x$, we have $x = (A+I)^{-1}y$. 
From the proved inequality, it follows 
$$
\Vert y\Vert  = \Vert (A+I)x\Vert  \ge 2\Vert x\Vert  = 2\Vert (A+I)^{-1}y\Vert 
$$
Rearranging yields 
$$\Vert(A+I)^{-1}y\Vert \le \frac{1}{2}\Vert y\Vert  \quad \implies \quad \Vert (A+I)^{-1}\Vert \le \frac{1}{2}.
$$
\\

\textit{Step 3.} We will show that $(A+I)^{-1}$ can be approximated in the operator norm by a sequence of finite-rank operators $P_N(A+I)^{-1}$. 

Let $P_N$ denote the orthogonal projection onto
$E_N=\operatorname{span}\{e_1,\ldots,e_N\}$, so that 
$$
\Vert(I - P_N)x\Vert^2 = \sum_{i > N} \vert x_i\vert^2.
$$
It follows from Step 2 that
$$
\langle Ax, x \rangle \ge \vert x_1\vert^2 + \sum_{i=2}^\infty (i-1)\vert x_i\vert^2 
$$
and therefore
$$
 \langle (A+I)x, x \rangle \ge 2\vert x_1\vert^2 + \sum_{i=2}^\infty i\vert x_i\vert^2.
 $$
For $y \in \ell^2(\mathbb{N})$ and $x = (A+I)^{-1}y$, we established the bound $\Vert x\Vert \le \frac{1}{2}\Vert y\Vert$. Using the Cauchy-Schwarz inequality, we obtain
\be\label{eq bound A+I}
\langle (A+I)x, x \rangle \le \Vert y\Vert \Vert x\Vert  \le \frac{1}{2}\Vert y\Vert^2. 
\ee
On the other hand,
$$
\langle (A+I)x, x \rangle  \ge \sum_{i=1}^\infty (i+1)\vert x_i\vert^2  
\ge \sum_{i > N} (i+1)\vert x_i\vert ^2 \ge  (N+2) \sum_{i > N} \vert x_i\vert^2.
$$
Combining with \eqref{eq bound A+I}, we have 
$$
(N+2) \sum_{i > N} \vert x_i\vert^2 \le \frac{1}{2}\Vert y\Vert^2 
\implies \sum_{i > N} \vert x_i\vert^2 \le \frac{1}{2(N+2)}\Vert y\Vert^2
$$
or 
$$
\Vert(I - P_N)(A+I)^{-1}y\Vert  \le \frac{1}{\sqrt{2(N+2)}}\Vert y\Vert.
$$
It follows that
$$
\lim_{N \to \infty} \Vert(A+I)^{-1} - P_N(A+I)^{-1}\Vert = 
\lim_{N \to \infty} \Vert(I - P_N)(A+I)^{-1}\Vert = 0,
$$
so that $(A+I)^{-1}$ is compact.
\\

\textit{Step 4.} It has been proved in the previous steps that $(A+I)^{-!}$ is an injective self-adjoint and compact operator. By the spectral theorem, there exists an orthonormal basis $\{u_j\}_{j \ge 1}$ of $\ell^2(\mathbb{N})$ such that:
$$
(A+I)^{-1} u_j = r_j u_j \quad \text{with } r_j \neq 0
$$
and $\lim_{j \to \infty} r_j = 0$. Since $0 < (A+I)^{-1} \le \frac{1}{2}I$, all eigenvalues are in the interval $(0, \frac{1}{2}]$. Then 
$$
A u_j = \lambda_j u_j, \quad \ \text{where } \lambda_j = 
\frac{1}{r_j} - 1.
$$
Moreover, $\lambda_j > 1$ and $\lambda_j \to \infty$ as $j \to \infty$.
\end{proof}

The next question in our study of the Bratteli diagram $B(F)$ is to find a lower bound for the growth of heights of the Kakutani-Rokhlin towers. 
In other words, we need a uniform lower bound for the set $\{H_j^{(n)}: j \in \N\}$.


\begin{lemma}\label{lem bound for heights} For $H_j^{(n)} = \sum_{i \in V_0} (F^n)_{j,i}$, we have
$$
\min_{j\in V_n} H_j^{(n)} \ge \left(\frac{n}{4}\right)^{n/2}.
$$
\end{lemma}

\begin{proof}
It is more convenient to work with the matrix $A = F^T$ because $F$ is symmetric. 

Clearly, $H_j^{(n)} = \sum_{i=1}^\infty (A^n)_{ij} \ge (A^n)_{jj}$.
Fix $m \le n/2$ and consider the length-$n$ path that begins at $j \in V_0$ and ends at $j \in V_n$:
$$
j\to j+1\to\cdots\to j+m
\to\underbrace{j+m\to\cdots\to j+m}_{n-2m\text{ steps}}
\to j+m-1\to\cdots\to j.
$$
The structure of the diagram shows that there are single paths from
$j \in V_0$ to $j+m \in V_m$ and from $j+m  \in V_{n-m}$ to $j\in V_n$.
Also, there are $(j+m+1)^{n-2m}$ finite paths connecting $j+m  \in V_{m}$ and $j+m  \in V_{n-m}$. This means that 
$$
H_j^{(n)}\ge(A^n)_{jj}\ge(j+m+1)^{n-2m}\ge(m+1)^{n-2m}.
$$
Take $m=\lfloor n/4\rfloor$. Then $m+1\ge n/4$ and $n-2m\ge n/2$, giving the claimed bound.
\end{proof}

\begin{lemma}\label{lem lambda bound}
For a fixed $\lambda > 1$, 
$$
\frac{\lambda^n}{(n/4)^{n/2}}\longrightarrow 0, \quad n \to \infty. 
$$
\end{lemma}

\begin{proof}
The result follows immediately from the relation
$$
\log\left(\frac{\lambda^n}{(n/4)^{n/2}}\right)
= n\log\lambda-\frac n2\log\frac n4\to-\infty, \quad n\to \infty.
$$
\end{proof}

We are ready to prove the main result of this section.

\begin{theorem}\label{thm no prob meas}
Let $B = B(F)$ be a stationary generalized Bratteli diagram defined by the symmetric matrix $F$; see \eqref{eq m-x F}. Then there is no tail invariant probability measure on the path space $X_B$. 
\end{theorem}

\begin{proof}
Assume that such a measure $\mu$ exists, and let $(p^{(n)})$ be the corresponding sequence of vectors assigned to the vertices of the diagram and satisfying the relation $p^{(n)}=Ap^{(n+1)}\ (A = F)$. Then
$p^{(0)}=A^np^{(n)}$ and the following holds:
\be\label{eq =1}
\sum_{i=1}^\infty H_i^{(n)}p_i^{(n)}=1, \quad \sum_{i=1}^\infty p_i^{(0)}=1, 
\qquad(n\ge 0)
\ee
Obviously, $p^{(0)}_i \le 1$ for all $i$ and $p^{(0)}\in\ell^2(\N)$. 

As was proved in Theorem \ref{thm A s-a}, there exists an orthonormal basis $(\xi^{(k)})$ formed by eigenvectors of $A$, $A\xi^{(k)} = \lambda_k \xi^{(k)}$ where $\lambda_k \to \infty$. Fix $k$ and abbreviate 
$\xi=\xi^{(k)}$, $\lambda=\lambda_k$. 

We first show the convergence of the double sum
$$
\begin{aligned}
\sum_{i,j}|\xi_i|(A^n)_{ij}p_j^{(n)}
&\le \|\xi\|_\infty\sum_j\left(\sum_i(A^n)_{ij}\right)p_j^{(n)}\\
&=\|\xi\|_\infty\sum_j H_j^{(n)}p_j^{(n)}\\
&=\|\xi\|_\infty<\infty.
\end{aligned}
$$
Using \eqref{eq =1}, the symmetry of $A^n$, and $A^n\xi=\lambda^n\xi$, we obtain
\be\label{eq inner}
\begin{aligned}
\langle\xi,p^{(0)}\rangle
&=\sum_{i,j}\xi_i(A^n)_{ij}p_j^{(n)}\\
&=\sum_j\left(\sum_i(A^n)_{ij}\xi_i\right)p_j^{(n)}\\
&=\lambda^n\sum_j\xi_jp_j^{(n)}.
\end{aligned}
\ee
Denoting $q_j^{(n)}:= H_j^{(n)}p_j^{(n)}$, we have a probability vector
$q^{(n)}$. Using Lemma \ref{lem bound for heights}, we write
\be\label{eq abs val}
\begin{aligned}
\left|\sum_j\xi_jp_j^{(n)}\right|
&=\left|\sum_j\xi_j\frac{q_j^{(n)}}{H_j^{(n)}}\right|\\
&\le \|\xi\|_\infty\sum_j\frac{q_j^{(n)}}{H_j^{(n)}}\\
&\le \frac{\|\xi\|_\infty}{\min_j H_j^{(n)}}\\
&\le \|\xi\|_\infty\left(\frac n4\right)^{-n/2}.
\end{aligned}
\ee
Combining \eqref{eq inner} and \eqref{eq abs val}, we deduce
$$
|\langle\xi,p^{(0)}\rangle|
\le \|\xi\|_\infty\frac{\lambda^n}{(n/4)^{n/2}}.
$$
Since $k$ is fixed, $\lambda=\lambda_k$ is a fixed finite number, 
Lemma \ref{lem lambda bound} gives
$$
\frac{\lambda_k^n}{(n/4)^{n/2}}\to 0,\quad n\to \infty.
$$
Therefore, for every $k$
$$
\langle\xi^{(k)},p^{(0)}\rangle= 0.
$$
Hence $p^{(0)}$ must be a zero vector, a contradiction.

\end{proof}

Analyzing the proof of Theorem \ref{thm no prob meas}, we see that this result can be obtained for a class of stationary Bratteli diagrams that satisfy the following conditions.

\begin{theorem}
Let $B = B(F)$ be a stationary generalized Bratteli diagram, with $A=F^T$ realized as a self-adjoint operator acting on $\ell^2(V_0)$. Assume that $A$ has an orthonormal basis $\{\xi^{(k)}\}$ of eigenvectors with finite eigenvalues $\lambda_k$. Define $M_n := \inf_{v\in V_n} H_v^{(n)}$. If 
for every fixed $k$,
\be\label{eq growth cond}
\frac{\lambda_k^n}{M_n} \to 0, \quad n\to \infty, 
\ee
then $B$ admits no tail invariant probability measure.
\end{theorem}

Indeed, the proof of Theorem \ref{thm no prob meas} gives the estimate
$$
|\langle \xi^{(k)}, p^{(0)} \rangle \le ||\xi||_\infty \frac{\lambda_k^m}{M_n}.
$$
Since $p^{(0)}\in \ell^2$, we obtain that $p^{(0)} = 0.$

This formulation makes the real mechanism transparent: if the uniform tower growth is faster than every individual spectral exponential, then the diagram has no tail invariant probability measure.

A clean growth condition that implies the criterion in \eqref{eq growth cond} is
$$
\frac{\log M_n}{n} \to \infty, \quad n \to \infty.
$$

\textbf{Acknowledgments.} The authors thank our colleagues and coauthors P. Jorgensen, O. Karpel, J. Kwiatkowski, T. Raszeja, and S. Sanadhya for numerous discussions during the development of this project. Special thanks go to the members of the Operator Theory seminar at the University of Iowa, where the results of this paper were presented. 

\textbf{AI Statement.} No AI tool was used or prompted at any point in any capacity in this paper.

\bibliographystyle{alpha}
\bibliography{ReferencesProbMeas}

\end{document}